\documentclass[11pt, oneside]{amsart}
\usepackage[text={6.5in,9in},centering,letterpaper,dvips]{geometry}
\usepackage{graphicx}
\usepackage{amsfonts}
\usepackage[dvipsnames]{xcolor}
\usepackage{subcaption}
\usepackage{epsf}
\usepackage{amssymb}
\usepackage{amsmath}
\usepackage{amscd}
\usepackage{pdfpages}
\usepackage{fancyhdr}
\usepackage{setspace}
\usepackage{soul} 
\usepackage[all]{xy}
\usepackage{verbatim}
\usepackage{enumerate}
\usepackage[colorlinks=true, urlcolor=NavyBlue, linkcolor=NavyBlue, citecolor=NavyBlue,backref=page]{hyperref}
\usepackage{tikz}
\usetikzlibrary{cd, calc, positioning, fit, arrows, decorations.pathreplacing, decorations.markings, shapes.geometric, backgrounds,bending}
\usepackage{adjustbox}
\usepackage{tikzsymbols}
\usepackage{mdframed}
\usepackage{wrapfig}

\renewcommand*{\backref}[1]{}
\renewcommand*{\backrefalt}[4]
{%
    \ifcase #1 (Not cited.)%
        \or        (Cited on page~#2.)
        \else      (Cited on pages~#2.)
    \fi
}

\makeatletter

\renewcommand{\p@subfigure}{\thefigure.} 
\makeatother

\theoremstyle{theorem}
\newtheorem{theorem}{Theorem}
\newtheorem{proposition}[theorem]{Proposition}

\newtheorem{question}[theorem]{Question}
\newtheorem{corollary}[theorem]{Corollary}
\newtheorem{conjecture}[theorem]{Conjecture}

\makeatletter
\newtheorem*{rep@theorem}{\rep@title}
\newcommand{\newreptheorem}[2]{%
\newenvironment{rep#1}[1]{%
 \def\rep@title{#2 \ref{##1}}%
 \begin{rep@theorem}}%
 {\end{rep@theorem}}}
\makeatother

\newreptheorem{theorem}{Theorem}
\newreptheorem{lemma}{Lemma}
\newreptheorem{question}{Question}
\newreptheorem{corollary}{Corollary}
\newreptheorem{proposition}{Proposition}

\theoremstyle{definition}
\newtheorem{definition}[theorem]{Definition}
\newtheorem{remark}[theorem]{Remark}

\newcommand{\Z}{\mathbb{Z}}

\newcommand{\Aa}{\mathcal A}

\newcommand{\Kk}{\mathcal K}

\newcommand{\Pp}{\mathcal P}

\newcommand{\Ss}{\mathcal S}
\newcommand{\Tt}{\mathcal T}
\newcommand{\Uu}{\mathcal U}

\newcommand{\Yy}{\mathcal Y}

\makeatletter
\def\@seccntformat#1{%
  \protect\textup{\protect\@secnumfont
    \ifnum\pdfstrcmp{subsection}{#1}=0 \bfseries\fi
    \csname the#1\endcsname
    \protect\@secnumpunct
  }%
}  
\makeatother

\begin{document}

\rhead{\thepage}
\lhead{\author}
\thispagestyle{empty}


\raggedbottom
\pagenumbering{arabic}
\setcounter{section}{0}


\title{Indecomposable Klein bottles with order-4 meridians}

\author{Jeffrey Meier}
\address{Department of Mathematics, Western Washington University, Bellingham, WA 98225}
\email{jeffrey.meier@wwu.edu}
\urladdr{http://jeffreymeier.org} 

\begin{abstract}
	We exhibit an infinite family of indecomposable Klein bottles in the 4--sphere with order-4 meridians.
\end{abstract}

\maketitle

\section{Introduction}
\label{sec:intro}

A foundational theorem in knot theory states that every knot in $S^3$ can be expressed uniquely as a connected sum of nontrivial prime knots.
It is natural to wonder to what extent such a theorem might hold in dimension four, where one is concerned with embeddings of surfaces in $S^4$, or \emph{surface-knots}.
To this day, very little is known about the ways in which a surface-knot might decompose as the connected sum of simpler surfaces.
For example, each of the following questions remains unanswered:
\begin{enumerate}
	\item \emph{Can the unknotted 2--sphere be expressed as the connected sum of two knotted 2--spheres?}
	\item \emph{Is every knotted projective plane the connected sum of a knotted 2--sphere and an unknotted projective plane?}
\end{enumerate}
The first of the above questions is closely related to the famous Unknotting Conjecture for 2--knots~\cite[pages 55, 97]{CarKamSai_04_Surfaces-in-4-space}, while the second is the purview of the Kinoshita Conjecture, which posits an affirmative answer~\cite[Remark~3.7]{KatSae_98_Embeddings-of-quaternion-space}.

\begin{wrapfigure}[15]{r}{0.43\textwidth}
\vspace{-7mm}
\begin{mdframed}
	\centering
	\includegraphics[width=\textwidth]{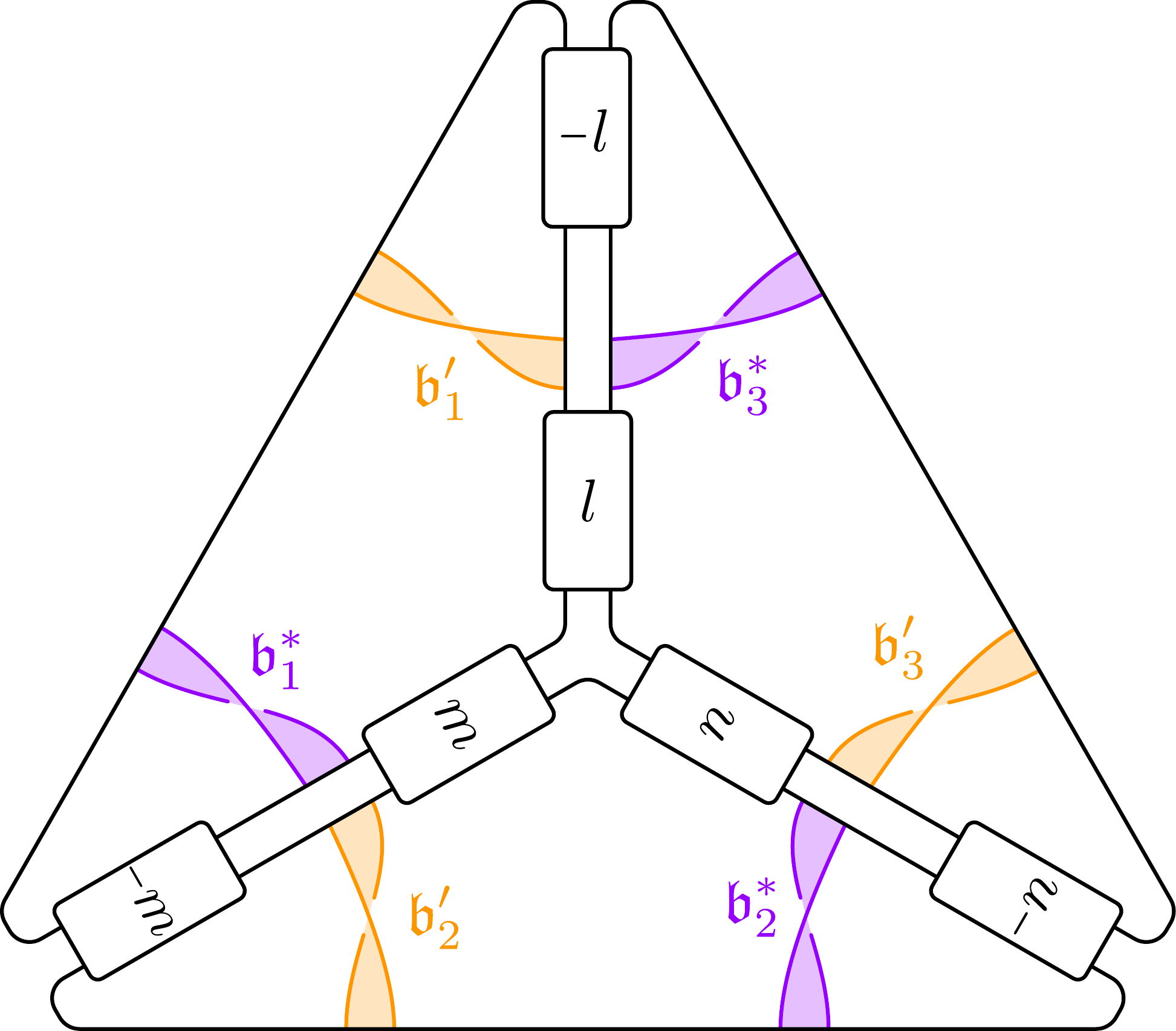}
	\caption{The Klein bottle $\Kk$}
	\label{fig:bottle}
\end{mdframed}
\end{wrapfigure}

It is known that surface-knots can decompose in surprising ways: Viro gave an example of a knotted 2--sphere $\Ss$ such that $\Ss\#\Pp = \Pp$, where $\Pp$ is the unknotted projective plane~\cite{Vir_73_Local-knotting}; see also~\cite[page 56]{CarKamSai_04_Surfaces-in-4-space}.
This illustrates the subtlety of finding an analogue in dimension four to the prime factorization theorem for knots in dimension three, especially where nonorientable surface-knots are concerned.

A surface-knot $\Ss\subset S^4$ is called \emph{decomposable} if we have $\Ss = \Ss_1\#\Ss_2$ such that neither $\Ss_1$ nor $\Ss_2$ is a 2--sphere; otherwise, $\Ss$ is called \emph{indecomposable}.

The purpose of this paper is to present an infinite family of indecomposable Klein bottles with order-4 meridians.
Let $l$ be an even integer with $|l|\geq 2$, and let $m$ and $n$ be odd integers with $|m|,|n|\geq 3$.
Let $\Kk(l,m,n)$ be the Klein bottle introduced in Definition~\ref{def:klein_bottles} below and shown in Figure~\ref{fig:bottle}.

\begin{theorem}
\label{thm:main}
	If one of $m$ or $n$ is divisible by 3 and the other is divisible by 3 or 5, then the Klein bottle $\Kk(l,m,n)$ has order-4 meridians.
	If $m=n$, then $\Kk(l,m,m)$ is indecomposable.
\end{theorem}

The first part of Theorem~\ref{thm:main} is Corollary~\ref{cor:meridians}, while the second part is Proposition~\ref{prop:indecomp}.
Corollary~\ref{cor:distinct} shows that the $\Kk$ are pairwise non-isotopic whenever their parameters are distinct modulo signs.

Indecomposable Klein bottles have previously been given by Yoshikawa~\cite{Yos_98_The-order-of-a-meridian-of-a-knotted} and by Lidman and Piccirillo~\cite{LidPic_25_Stably-irreducible-non-orientable}.
In fact, we use the argument of Lidman and Piccirillo to prove Proposition~\ref{prop:indecomp}.
Interestingly, the double branched cover of our Klein bottle $\Kk(2,3,7)$ has the same fundamental group as the double branched cover of their Klein bottles.

The first example of a Klein bottle with order-4 meridians was given by Yoshikawa, who produced, for each even (nonnegative) integer, a Klein bottle having meridians of that order~\cite{Yos_98_The-order-of-a-meridian-of-a-knotted}.
Curiously, the Klein bottle $\Yy$ with order-4 meridians produced by Yoshikawa has the same group as our Klein bottle $\Kk(2,3,3)$.
Yoshikawa's Klein bottles with meridians whose order is not 2 or 4 are automatically indecomposable:
The order of a meridian of a projective plane (or any connected sum of projective planes) is either 2 or 4, since the peripheral subgroup is the quaternion group $Q_8$~\cite{Pri_77_Homeomorphisms-of-quaternion-space}.

Surface-knots $\Ss$ and $\Ss'$ are \emph{stably isotopic} if there exist unknotted surface-knots $\Uu$ and $\Uu'$ such that $\Ss\#\Uu$ and $\Ss'\#\Uu'$ are isotopic.
A surface-knot $\Ss$ is \emph{stably reducible} if it is stably isotopic to a surface-knot $\Ss'$ with $\chi(\Ss')>\chi(\Ss)$; otherwise, it is \emph{stably irreducible}.

\begin{proposition}
\label{prop:stably}
	Let $\Kk$ be any $\Kk(l,m,n)$.
	Then, either
	\begin{enumerate}
		\item $\Kk$ is stably irreducible, or
		\item $\Kk$ is stably isotopic to a projective plane with order-4 meridians.
	\end{enumerate}
\end{proposition}

Note that projective planes with order-4 meridians are not known to exist~\cite[Section~4]{Yos_98_The-order-of-a-meridian-of-a-knotted}; the existence of one would give a negative answer to the Kinoshita Conjecture, which was discussed at the outset of this article.

The first examples of stably irreducible surface-knots were orientable and were given by Livingston~\cite{Liv_85_Stably-irreducible}; more recently, Lidman and Piccirillo gave examples of stably irreducible non-orientable surface-knots~\cite{LidPic_25_Stably-irreducible-non-orientable}.
The Klein bottles of Lidman and Piccirillo are also indecomposable, but it is not known whether they have order-4 meridians.
The Klein bottles of Lidman and Piccirillo are ribbon, and it is not clear whether the Klein bottles produced in this note are ribbon.

In Section~\ref{sec:future}, we give a number of open questions motivated by this work.
In particular, Conjecture~\ref{conj:order-4} posits that all of the $\Kk(l,m,n)$ are indecomposable, have order-4 meridians, and are stably irreducible.

\subsection*{Acknowledgements}
The author is grateful to James Dix for his participation in an REU project in 2016 hosted by Indiana University during which time these Klein bottles were first explored.
Sincere thanks are due to Cameron Gordon and Tye Lidman for insightful comments on an earlier draft of this article.
This work was supported by NSF grant DMS-2405324.

\section{The Klein bottles}
\label{sec:bottles}

Let $l$ be an even integer with $|l|\geq 2$, and let $m$ and $n$ be odd integers with $|m|\geq 3$ and $|n|\geq 3$.
Let $K = P(l,m,n)$ denote the corresponding 3--stranded pretzel knot; see Figure~\ref{fig:knot}, where each twist box encodes the indicated number of positive or negative half-twists.
Importantly, $K$ is strongly invertible via the (green) axis of rotation $C$ shown in Figure~\ref{fig:knot}.

\begin{figure}[h!]
	\begin{subfigure}{.22\textwidth}
		\centering
		\includegraphics[width=.8\linewidth]{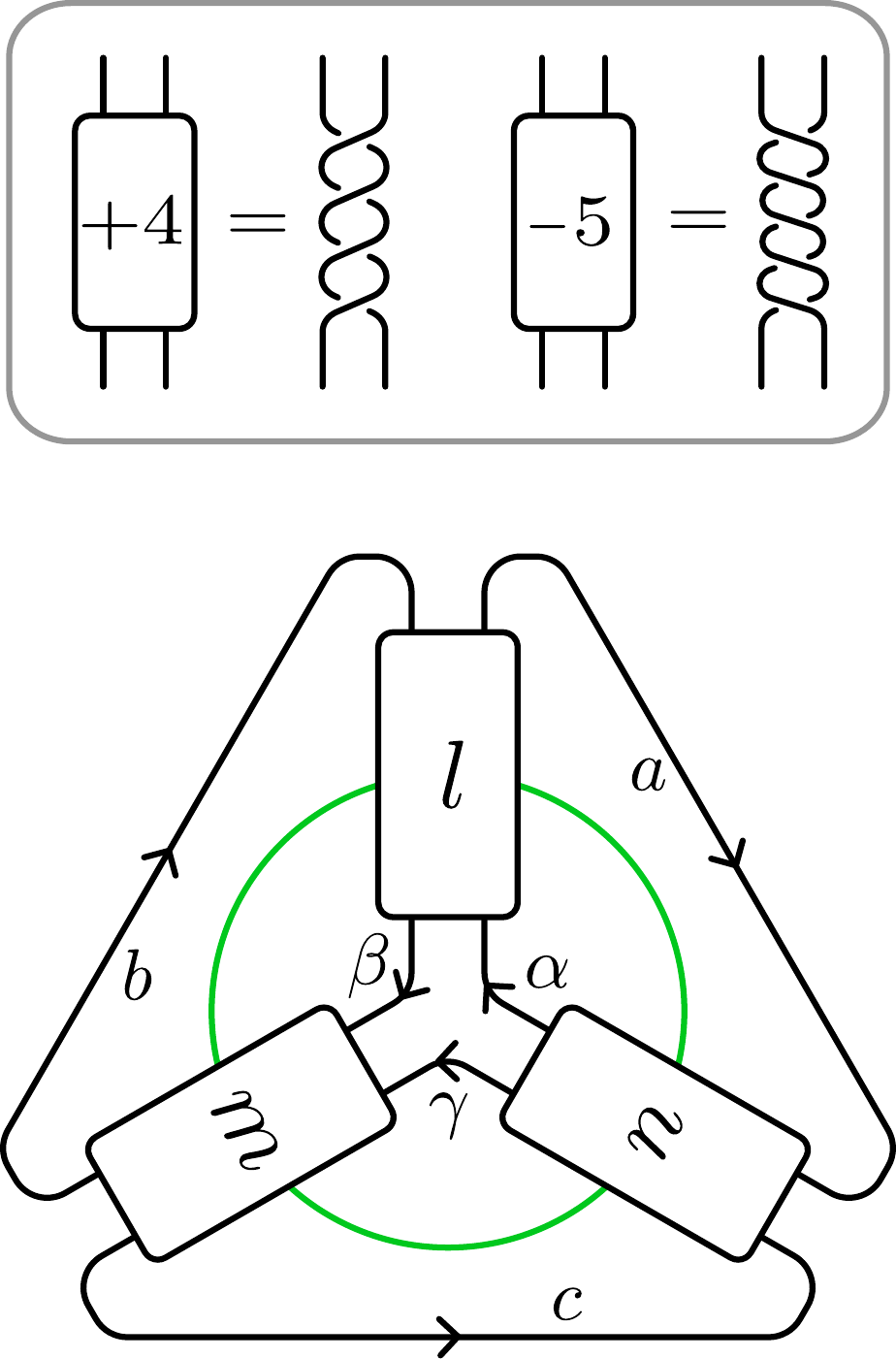}
		\caption{The pretzel knot $K$}
		\label{fig:knot}
	\end{subfigure}%
	\begin{subfigure}{.39\textwidth}
		\centering
		\includegraphics[width=.8\linewidth]{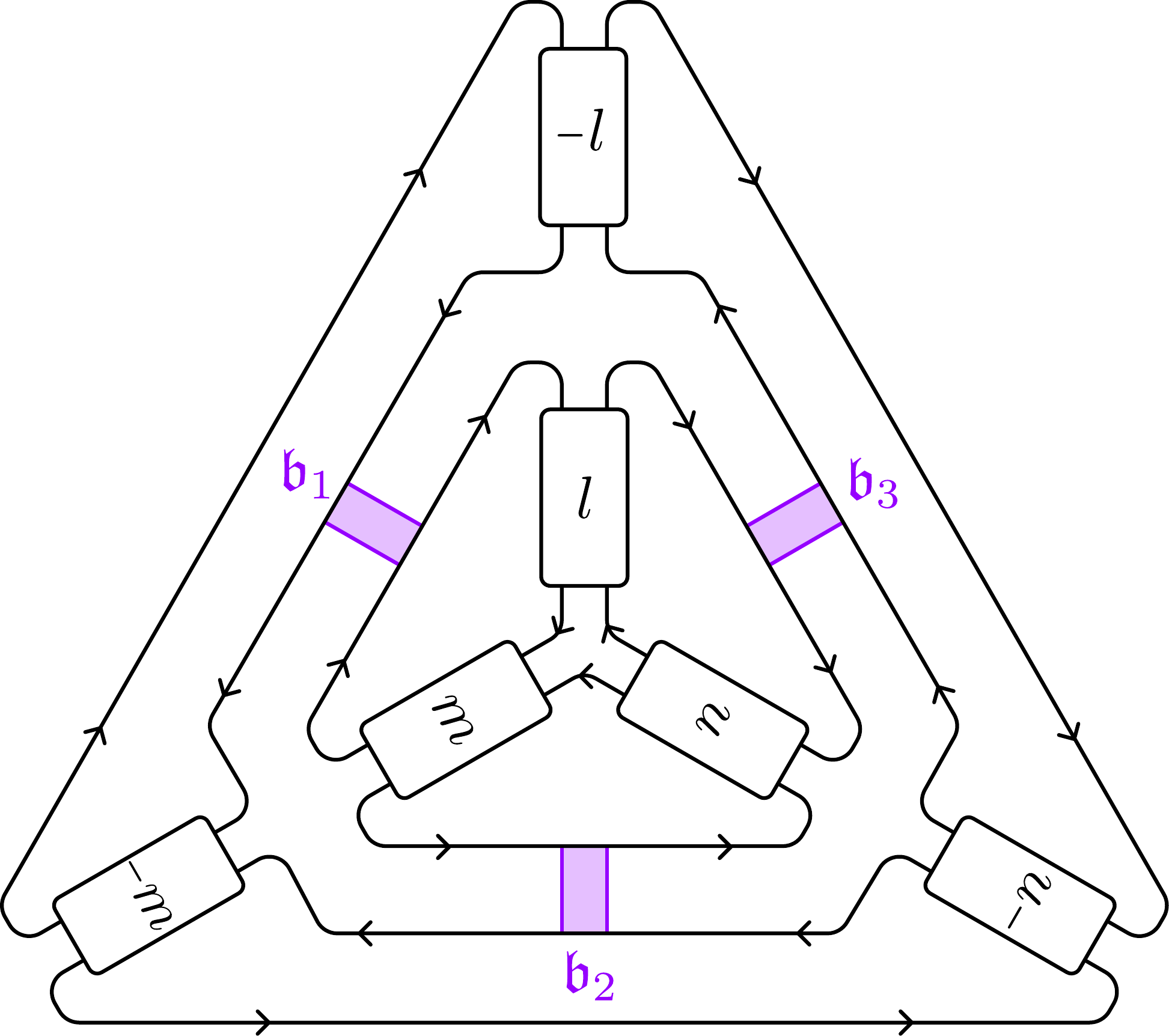}
		\caption{The ribbon annulus $\Aa$}
		\label{fig:annulus_1}
	\end{subfigure}%
	\begin{subfigure}{.39\textwidth}
		\centering
		\includegraphics[width=.8\linewidth]{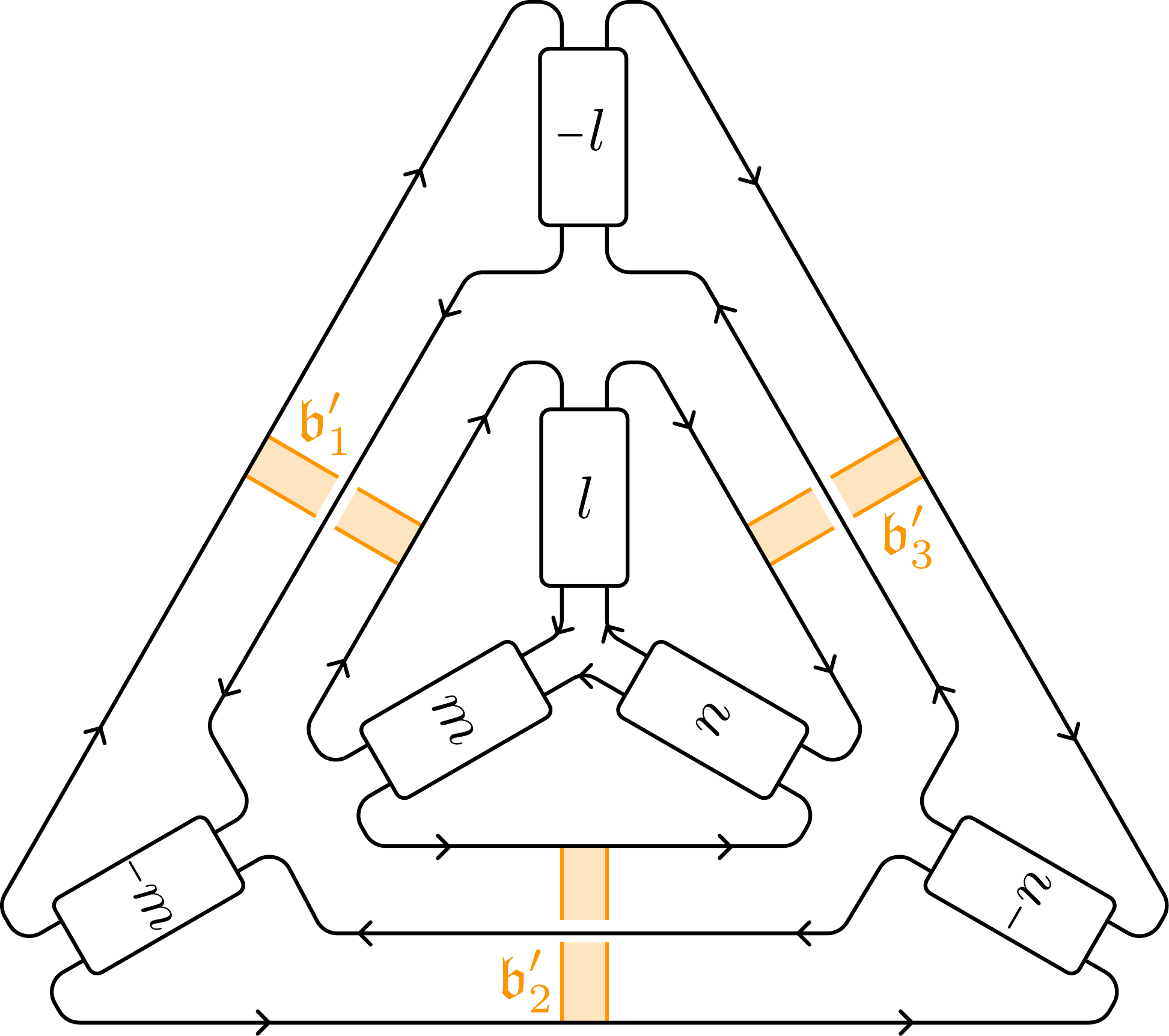}
		\caption{The ribbon annulus $\Aa'$}
		\label{fig:annulus_2}
	\end{subfigure}%
	\caption{Constructing the Klein bottles}
	\label{fig:construction}
\end{figure}

Regarding $K$ as a subset of $B^3$, we can form the product annulus $(B^4,\Aa) = (B^3,K)\times [0,1]$.
Note that
$$\partial(B^4,\Aa) = (B^3,K)\cup_{\partial B^3}\overline{(B^3,K)} = (S^3,K\sqcup \overline K),$$
where $\sqcup$ denotes split union and $\overline K$ is the mirror reverse of $K$.
A ribbon presentation for $\Aa$ with bands $\frak b_1$, $\frak b_2$, and $\frak b_3$ is shown in Figure~\ref{fig:annulus_1}, where $\partial B^3$ manifests as the 2--sphere splitting the inner knot $K$ from the outer knot $\overline K$.

Let $\tau\colon B^3\to B^3$ denote the involution of $K$ given by rotation through $\pi$ radians about the green circle $C$ shown in Figure~\ref{fig:knot}.
For $t\in[0,1]$, let $\tau_t\colon B^3\to B^3$ denote rotation through $\pi t$ radians about $C$, so $\tau_t$ is an isotopy from the identity to $\tau$.
Let $\Tt\colon B^3\times[0,1]\to B^3\times [0,1]$ given by $\Tt(x,t) = (\tau(x),t)$.

Now consider the annulus $(B^4, \Aa') = \Tt(B^4,\Aa)$.
A ribbon presentation for $\Aa'$ with bands $\frak b_1'$, $\frak b_2'$, and $\frak b_3'$ is shown in Figure~\ref{fig:annulus_2}.
Note that $\Aa'$ is isotopic to $\Aa$ via an ambient isotopy that is supported to the outside of $\partial B^3$ and that inverts the outer knot $\overline K$.

\begin{definition}
\label{def:klein_bottles}
	Let $\Kk = \Kk(l,m,n)$ denote the Klein bottle in $S^4$ defined by
	$$(S^4, \Kk) = (B^4,\Aa)\cup_{(S^3,K\sqcup\overline K)}\overline{(B^4,\Aa')}.$$
Throughout the paper, we will reserve the notation $\Kk$ for a member of this family of Klein bottles, suppressing the parameters when reasonable.
\end{definition}

A handle-decomposition of $\Kk$ is shown in Figure~\ref{fig:bottle}: The figure shows an unlink of three components with six bands attached.
The result of resolving these six bands is another unlink of three components.
Thus, the handle-decomposition has three minima, six saddles, and three maxima.

This handle-decomposition is derived from the definition as follows:
First, combine Figures~\ref{fig:annulus_1} and~\ref{fig:annulus_2} so all six bands are present, as in Figure~\ref{fig:diffeo1}; the bands can be freely isotoped along the split link to which they are attached until they are disjoint.
Next, resolve the $\frak b_i$, recording the dual bands $\frak b_i^*$.
Finally, slide each $\frak b_i^*$ over its corresponding $\frak b_i'$.

\begin{remark}
	The construction outlined above is a special instance of a more general mapping torus construction, which we describe here for completeness; it is inspired by an analogous construction of Price and Roseman for constructing projective planes~\cite{PriRos_75_Embeddings-of-the-projective-plane}.
	
	Let $K$ be an oriented knot and let $\Phi_t\colon S^3\times I\to S^3$ be an ambient isotopy taking $K$ to its reverse~$-K$. 
	This data gives rise to an embedded Klein bottle $\widehat{\Kk}(K,\Phi)$ in $S^3\times S^1$, given by
	$$\widehat{\Kk}(K,\Phi)\cap (S^3\times\{t\}) = \Phi_t(K).$$
	Let $\omega$ be a simple closed curve in $S^3\times S^1$, disjoint from $\widehat{\Kk}(K,\Phi)$ with the property that $[\omega]$ is a generator of $H_1(S^3\times S^1)\cong\Z$.
	Since $\omega$ is one dimensional, there is an ambient isotopy taking $\omega$ to a core curve $\{x\}\times S^1$.
	By performing surgery on $\omega$, we can transform $S^3\times S^1$ into $S^4$.
	Since $\omega$ is disjoint from $\widehat{\Kk}(K,\Phi)$, the image of $\widehat{\Kk}(K,\Phi)$ is an embedded Klein bottle in $S^4$, which we denote $\Kk(K,\Phi,\omega)$ and call the \emph{flip-spin} of $K$ with respect to $\Phi$ and $\omega$.
\end{remark}

In order to obstruct decomposability using the method of Lidman and Piccirillo, we will need to establish that a subclass of our Klein bottles (those with $n=m$) are amphicheiral in the sense that they are isotopic to their mirrors.
This is implied by the following.

\begin{proposition}
\label{prop:amphi}
	The Klein bottle $\Kk(l,m,n)$ is isotopic to the mirror of $\Kk(l,n,m)$.
\end{proposition}

\begin{proof}
	Let $\Kk = \Kk(l,m,n)$, and recall that
	$$(S^4,\Kk) = (B^4, \Aa)\cup \overline{\Tt(B^4, \Aa)}.$$
	The handle-decomposition corresponding to this splitting is shown in Figure~\ref{fig:diffeo1}, with the purple bands $\frak b_i$ describing $\Aa$ and the orange bands $\frak b_i'$ describing $\Tt(\Aa)$.
	If we apply $\Tt^{-1}$ to each $B^4$, this gives an ambient isotopy of $S^4$ taking $(S^4,\Kk)$ to $\Tt^{-1}(B^4,\Aa)\cup\overline{(B^4,\Aa)}$.
	The result of this ambient isotopy is shown in Figure~\ref{fig:diffeo2}.
	
	Next, we apply the ambient isotopy given by rotating each time-slice $S^3\times\{t\}$  of $S^4$ through $\pi$ radians about the vertical axis of the page; this is the rotation that leaves invariant the $(\pm l)$--twist boxes and swaps the $(\pm m)$--twist boxes with the $(\pm n)$--twist boxes.
	The result of this ambient isotopy is shown in Figure~\ref{fig:diffeo3}.
	
	Finally, we apply the orientation-reversing diffeomorphism given by reflection across the equatorial time-slice $S^3\times\{0\}$.
	This has the effect of swapping the colors of the bands in Figure~\ref{fig:diffeo3}.
	The resulting Klein bottle is the mirror of $\Kk(l,n,m)$, since we have applied an orientation-reversing diffeomorphism.
\end{proof}

\begin{figure}[h!]
	\begin{subfigure}{.33\textwidth}
		\centering
		\includegraphics[width=.9\linewidth]{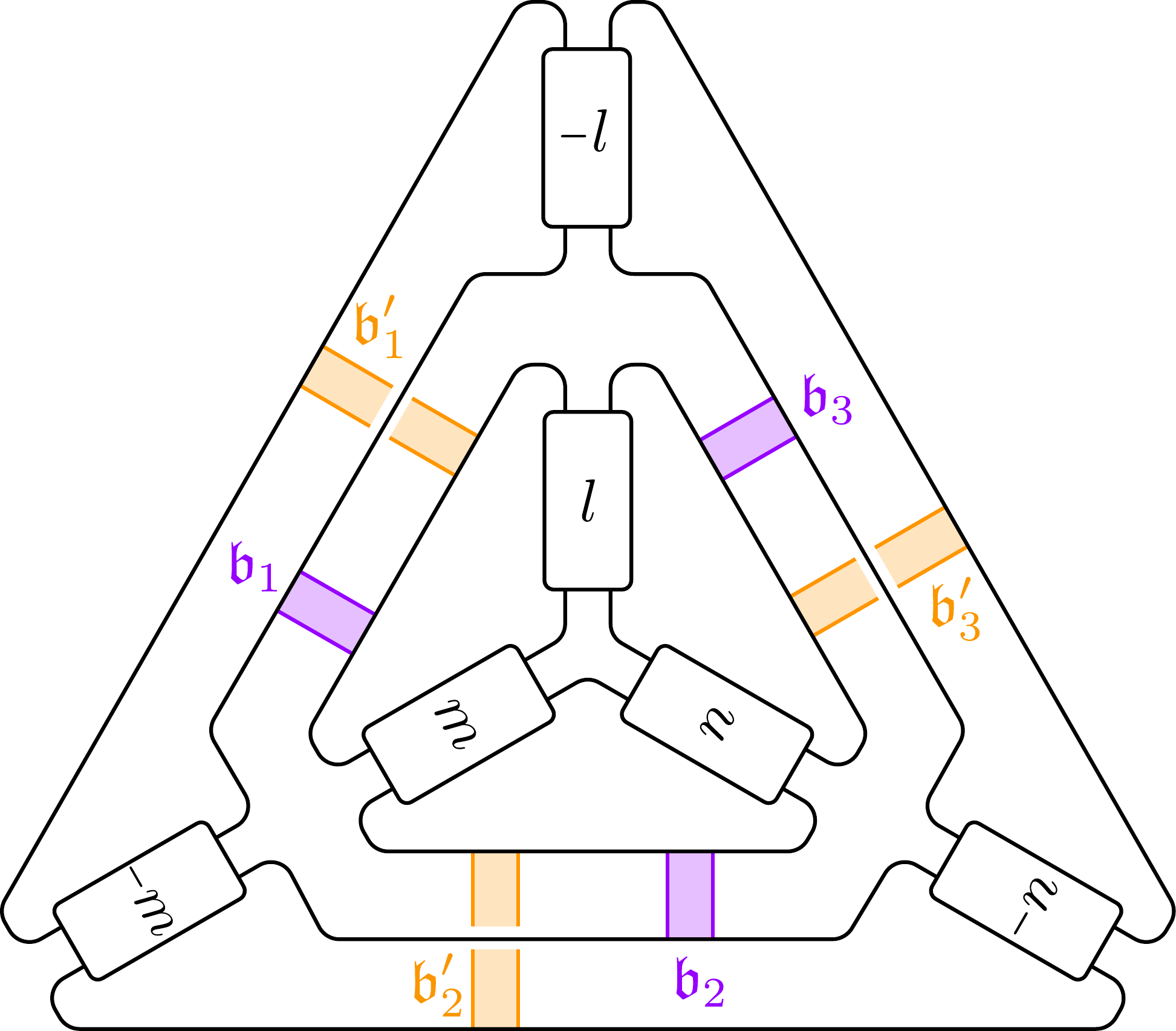}
		\caption{}
		\label{fig:diffeo1}
	\end{subfigure}%
	\begin{subfigure}{.33\textwidth}
		\centering
		\includegraphics[width=.9\linewidth]{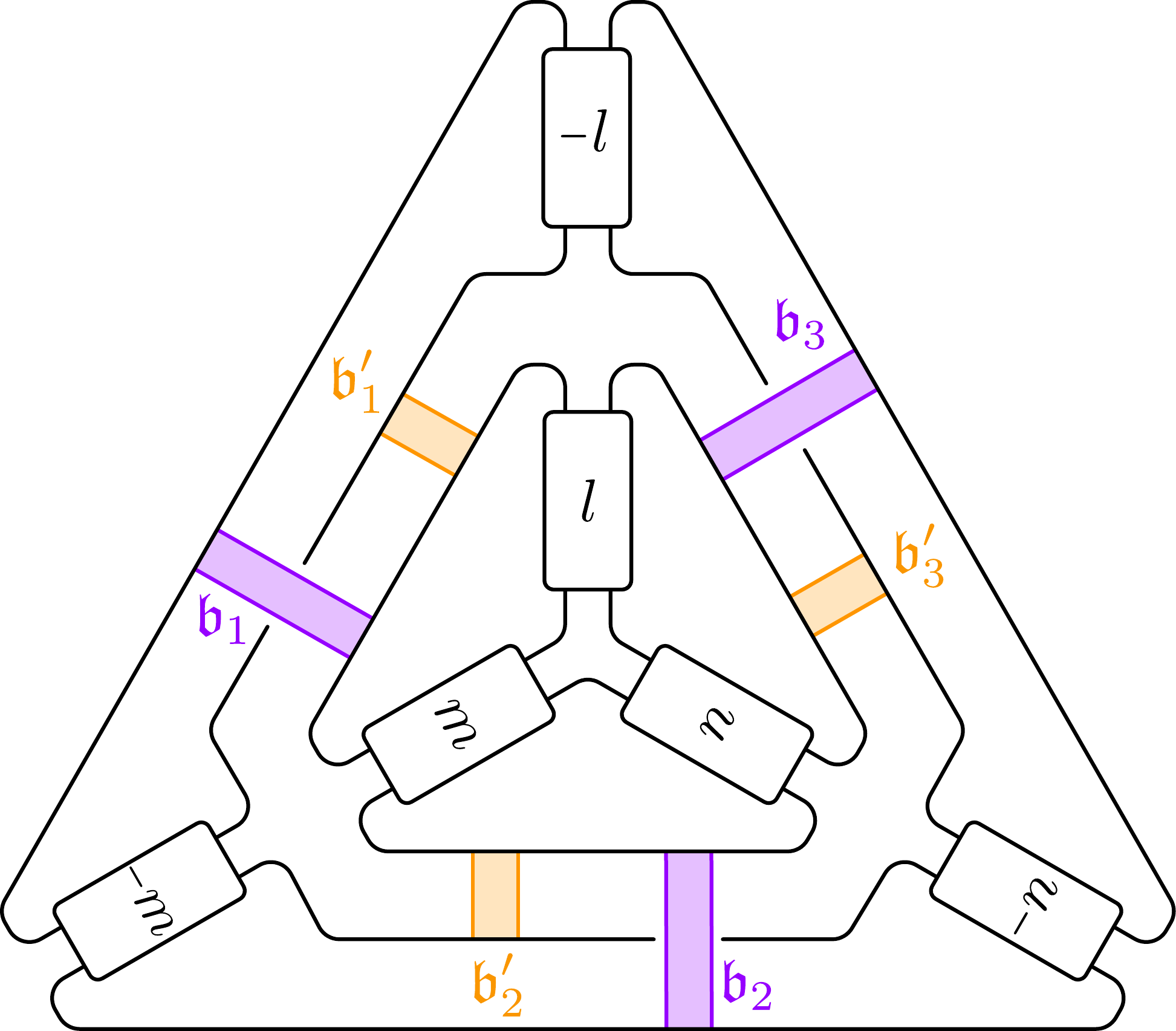}
		\caption{}
		\label{fig:diffeo2}
	\end{subfigure}%
	\begin{subfigure}{.33\textwidth}
		\centering
		\includegraphics[width=.9\linewidth]{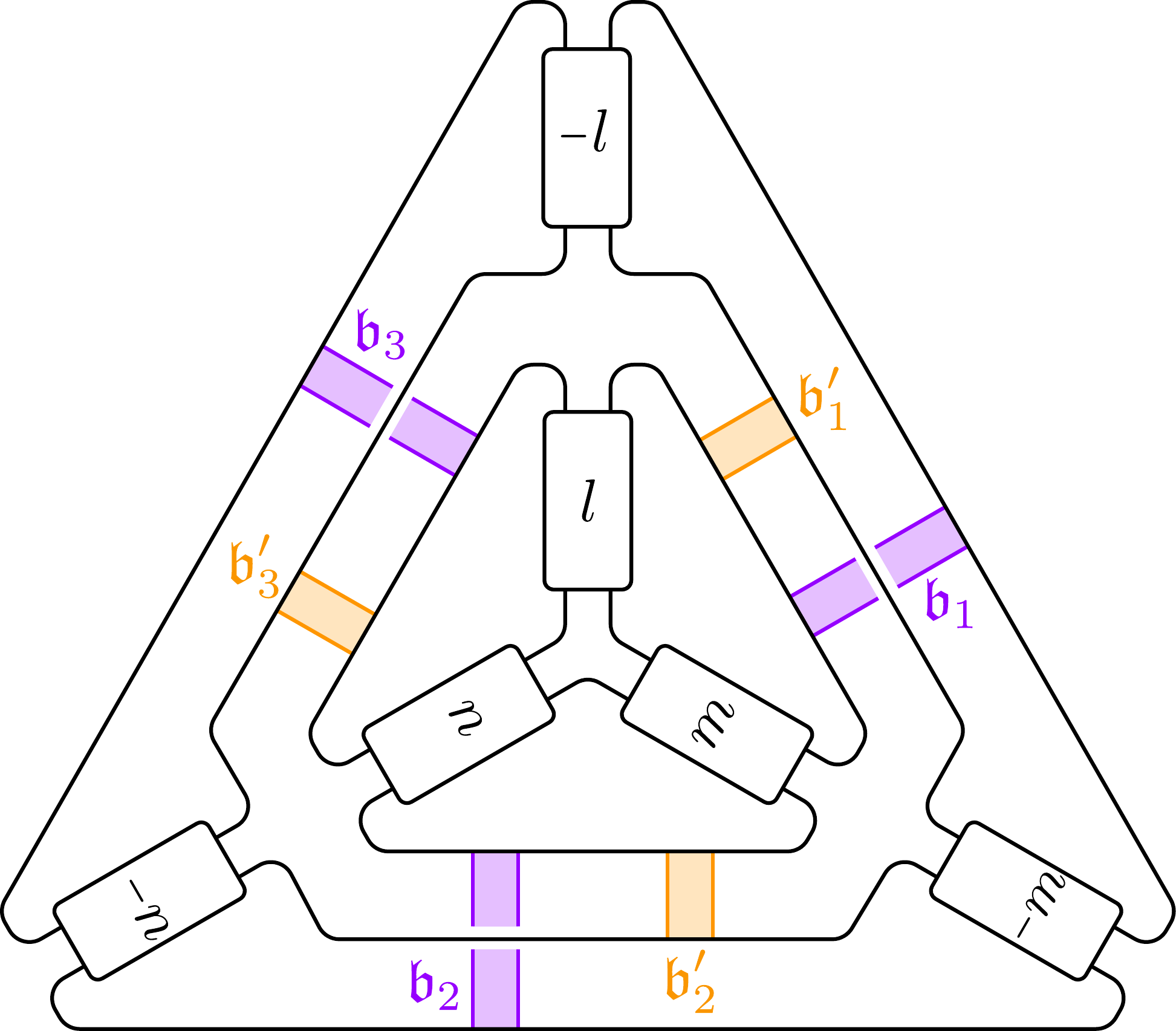}
		\caption{}
		\label{fig:diffeo3}
	\end{subfigure}%
	\caption{Three isotopic presentations of the Klein bottle $\Kk$.}
	\label{fig:diffeos}
\end{figure}

We note that Figure~\ref{fig:diffeo1} can be reflected across the vertical axis of the page to give $\Kk(-l,-m,-n)$.
It follows that $\Kk(-l,-m,-n)$ is the mirror of $\Kk(l,m,n)$.

\section{The groups}
\label{sec:groups}

We now commence our analysis of the group $\pi(\Kk)$.
Let $a$, $b$, and $c$ denote the meridional generators of $\pi(K)$ indicated in Figure~\ref{fig:knot}.
Using the natural inclusion map, we regard $a$, $b$, and $c$ as elements of $\pi(\Kk)$.

\begin{proposition}
\label{prop:group_K}
	The group of the Klein bottled $\Kk(l,m,n)$ in $S^4$ is presented as
	$$\pi(\Kk(l,m,n)) = \langle a, b, c\mid a^4 = b^4 = c^4 = (bc)^m = (ca)^n = 1, (ab)^l = a^2 = b^2 = c^2\rangle.$$
\end{proposition}

\begin{proof}
	Assume for now that $l$, $m$, and $n$ are all positive; we will see later that the result is independent of the signs of these parameters.
	Using the standard technique to derive a Wirtinger presentation for $\pi(K)$, we find $\pi(K)$ is generated by the six meridional elements $a$, $b$, $c$, $\alpha$, $\beta$, and $\gamma$ that are indicated in Figure~\ref{fig:knot}, subject to the following six relations 
	$$\begin{cases}
		\alpha = (ab^{-1})^{\frac{l}{2}}a(ba^{-1})^{\frac{l}{2}} \\ 
		\alpha = (c^{-1}a^{-1})^{\frac{n-1}{2}}c(ac)^{\frac{n-1}{2}} \\
		\beta = (ab^{-1})^{\frac{l-2}{2}}aba^{-1}(ba^{-1})^{\frac{l-2}{2}} \\
		\beta = (bc)^{\frac{m-1}{2}}bcb^{-1}(c^{-1}b^{-1})^{\frac{m-1}{2}} \\
		\gamma = (bc)^{\frac{m-1}{2}}b(c^{-1}b^{-1})^{\frac{m-1}{2}} \\
		\gamma = (c^{-1}a^{-1})^{\frac{n-1}{2}}c^{-1}ac(ac)^{\frac{n-1}{2}}
	\end{cases}$$
	
	The inclusion map $K\hookrightarrow\Aa$ induces an isomorphism, and the bands from $\Aa'$ contribute three more relations: $\alpha = a^{-1}$, $\beta = b^{-1}$, and $\gamma = c^{-1}$.
	These relations are equivalent to the relations claimed by the proposition as follows.
	
	Use the relations $\alpha = a^{-1}$, $\beta = b^{-1}$, and $\gamma = c^{-1}$ to eliminate the generators $\alpha$, $\beta$, and $\gamma$.
	Then, considering those of the above six relations involving only powers of $a$ and $b$, we get
	$$(b^{-1}a)^{\frac{l-2}{2}}b^{-1}a^{-1}ba = (ab^{-1})^{\frac{l-2}{2}} = (b^{-1}a)^{\frac{l-2}{2}}b^{-1}a^{-1}b^{-1}a^{-1},$$
	which implies $b^2 = a^{-2}$.
	This relation replaces the second of the six relations.
	
	Similarly, those of the six relations above involving only $b$ and $c$ give rise to $b^2 = c^{-2}$, which replaces fourth relation, while those of the six relations above involving only $c$ and $a$ give rise to $c^2 = a^{-2}$, which replaces the sixth relation.
	At this point, we see that $a^2 = b^{-2} = c^2 = a^{-2}$, so we deduce that $a^4=1$ and $a^2 = b^2 = c^2$.
	
	Using this, we can transform the first, third, and fifth of the six relations to $a^2 = (a^{-1}b)^l$, $(bc)^m = 1$, and $(ca)^n = 1$, respectively.
	Since $a^2$ is central, $a^4=1$, and $l$ is even, we can introduce $a^{2l}$ on the right side of the new first relation and distribute the $l$ copies of $a^2$ to arrive at $a^2 = (ab)^l$.
	
	It's not hard to see that negating the parameter $l$ has the effect of interchanging the roles of $a$ and $b$ (and of $\alpha$ and $\beta$) in those of the original six relations involving only these four letters.
	Independently negating $m$ or $n$ has the same effect on the the relations involving those parameters, as well.
	It follows that the final relations depend only on $|l|$, $|m|$, and $|n|$, as claimed at the start of the proof.
\end{proof}

We remark that an alternative calculation of $\pi(\Kk)$ can be given using the handle-decomposition of Figure~\ref{fig:bottle}.
It is not \emph{a priori} obvious whether $a$, $b$, and $c$ have order 2 or 4.
We suspect they always have order 4, but we have only been able to prove the following.

\begin{corollary}
\label{cor:meridians}
	If one of $m$ or $n$ is divisible by 3 and the other is divisible by 3 or 5, then the generators $a$, $b$, and $c$ of $\pi(\Kk(l,m,n))$ have order 4.
\end{corollary}

\begin{proof}
	Suppose without loss of generality that $3\mid m$ and $\delta\mid n$, where $\delta\in\{3,5\}$, and recall that $l$ is even.
	In this case, there is a surjection $\pi(\Kk(l,m,n))\twoheadrightarrow \pi(\Kk(2, 3, \delta))$.
	This latter group is finite and fits into the short exact sequence
	$$1\longrightarrow \text{SL}_2(\Z_\delta)\longrightarrow\pi(\Kk(2,3,\delta))\longrightarrow \Z_2\longrightarrow 1,$$
	which splits if $\delta=5$.
	These facts can be verified, for example, using GAP~\cite{GAP}.
	For either choice of $\delta$, one finds the image of $a$ has order 4, in $\pi(\Kk(2,3,\delta))$.
	It follows that $a$ has order 4 in $\pi(\Kk(l,m,n))$.
	Since $b$ and $c$ are conjugate to $a$, they have order 4, as well.
\end{proof}

Later, we will be interested in the quotient of $\pi(\Kk)$ obtained by killing the squares of the generators, which turns out to be a Coxeter group.
This is the group of the surface-knot obtained by taking the connected sum of $\Kk$ with an unknotted, nonorientable surface-knot.
Note that this quotient map is the identity if and only if the meridians of $\pi(\Kk)$ have order 2.

\begin{corollary}
\label{cor:coxeter}
	Let $\Uu$ be any unknotted, nonorientable surface-knot in $S^4$.
	Then,
	$$\pi(\Kk(l,m,n)\#\Uu) = \langle a, b, c\mid a^2 = b^2 = c^2 = (bc)^m = (ca)^n = (ab)^l = 1\rangle.$$
\end{corollary}

\begin{proof}
	The effect of taking the connected sum of $\Kk$ with $\Uu$ is to introduce to $\pi(\Kk)$ the relation $a^2=1$, which induces the relations $b^2=1$ and $c^2=1$, as well.
	This results in the claimed presentation.
\end{proof}

As an immediate corollary, we can appeal to the classification of finite Coxeter groups~\cite{Cox_35_The-Complete-Enumeration-of-Finite} to see that our family of Klein bottles contains infinitely many members that are pairwise distinct up to isotopy; see Question~\ref{ques:signs} regarding the converse to the following corollary.

\begin{corollary}
\label{cor:distinct}
	If $\{|l|,|m|,|n|\}\not=\{|l'|,|m'|,|n'|\}$, then $\Kk(l,m,n)$ is not isotopic to $\Kk(l',m',n')$.
\end{corollary}

To conclude this section, we calculate the fundamental group of the branched double cover $\Sigma_2(\Kk)$.
This group, which turns out to be a group known as a \emph{von Dyck group}, will be used in obstructing decomposability below.

\begin{corollary}
\label{cor:dbc_group}
	The group $\pi_1(\Sigma_2(\Kk))$ has the following presentation:
	$$\pi_1(\Sigma_2(\Kk(l,m,n))) = \langle u, v \mid u^l = v^m = (uv)^n = 1\rangle.$$
\end{corollary}

\begin{proof}
	Applying the Reidemeister-Schreier algorithm (as outlined by Fox~\cite{Fox_62_A-quick-trip}) to the presentation of $\pi(\Kk)$ coming from Proposition~\ref{prop:group_K}, we get the following presentation for the fundamental group of the double cover of the exterior $S^4\setminus\nu(\Kk)$:
	$$\langle a_1, a_2, b_1, b_2, c_1, c_2 \mid a_1a_2 = b_1b_2 = c_1c_2 = (a_1b_2)^l, (b_1c_2)^m = (c_1a_2)^n = (a_1a_2)^2 = 1 \rangle,$$
	where the $a_i$, $b_i$, and $c_i$ are associated to the (segment) lifts of $a$, $b$, and $c$, respectively.
	To obtain $\Sigma_2(\Kk)$ from this cover, we fill each meridian with a disk.
	This induces the extra relations $a_1a_2 = b_1b_2 = c_1c_2 = 1$, yielding
	$$\langle a_1, b_1, c_1, \mid (a_1b_1^{-1})^l = (b_1c_1^{-1})^m = (c_1a_1^{-1})^n = 1 \rangle,$$
	which becomes the desired presentation after the change of variable $u = a_1b_1^{-1}$ and $v = b_1c_1^{-1}$.
\end{proof}

\section{Decomposability and stable irreducibility}
\label{sec:decomp}

We can obstruct decomposability of a subfamily of our Klein bottles using the same argument as Lidman and Piccirillo~\cite[Remark]{LidPic_25_Stably-irreducible-non-orientable}.

\begin{proposition}
\label{prop:indecomp}
	If $m=n$, then $\Kk(l,m,m)$ is indecomposable.
\end{proposition}

\begin{proof}
	Let $\Kk = \Kk(l,m,n)$.
	In~\cite[Remark]{LidPic_25_Stably-irreducible-non-orientable}, Lidman and Piccirillo prove that their Klein bottle is indecomposable.
	Their argument applies to our Klein bottles, with the following three caveats.
	
	First, we need to know that $\Sigma_2(\Kk)$ has signature zero.
	This is implied by the fact that the normal Euler number of $\Kk$ is zero, which follows from Proposition~\ref{prop:amphi} in the case that $m=n$.
	(Even without the restriction that $m=n$, it is possible to verify that the normal Euler number of $\Kk$ is zero using the technique of~\cite[Section~3]{Kam_89_Nonorientable-surfaces-in-4-space}, which involves tracing longitudes for the unlink in Figure~\ref{fig:bottle} through the band attachments.)
	
	Second, we need to know that $\pi_1(\Sigma_2(\Kk))$ doesn't split as the free product of two nontrivial groups.
	If (up to signs) $(l,m,n)\not\in\{(2,3,3),(2,3,5)\}$, then this follows by the argument given by Lidman and Piccirillo: In this case, $\pi_1(\Sigma_2(\Kk))$ is generated by hyperbolic rotations of finite order, so the Cayley graph contains no cut vertex, in contrast to the Cayley graph of a free product.
	If (up to signs) $(l,m,n)$ is $(2,3,3)$ or $(2,3,5)$, then $\pi_1(\Sigma_2(\Kk))$ is $A_4$ or $A_5$, respectively.
	Since these groups are finite, they cannot be nontrivial free products.
	
	Third, we need to know that $\Sigma_2(\Kk)$ admits an orientation-reversing diffeomorphism.
	This is implied by the fact that $\Kk$ is amphicheiral when $m=n$, as indicated by Proposition~\ref{prop:amphi}.
\end{proof}

Finally, we prove the second result from the introduction, offering some restriction on the stable isotopy class of our Klein bottles.

\begin{repproposition}{prop:stably}
	Let $\Kk$ be any $\Kk(l,m,n)$.
	Then, either
	\begin{enumerate}
		\item $\Kk$ is stably irreducible, or
		\item $\Kk$ is stably isotopic to a projective plane with order-4 meridians.
	\end{enumerate}
\end{repproposition}

\begin{proof}
	If $\Kk$ is stably irreducible, we are done.
	Suppose instead that $\Kk\#\Uu = \Pp\#\Uu$, where $\Pp$ is a projective plane and $\Uu$ and $\Uu'$ are unknotted surfaces, which we can assume to be non-orientable by adding on additional projective plane summands.
	
	In Corollary~\ref{cor:coxeter}, we found that $G = \pi(\Kk\#\Uu)$ is a Coxeter group.
	Howlett calculated the second homology (or Schur multipliers) of such groups~\cite[Theorem~A]{How_88_On-the-Schur-multipliers-of-Coxeter}.
	In our case, the Coxeter group is described by a triangular graph with edge weights $l$, $m$, and $n$, with $l$ even and $m$ and $n$ odd; therefore, Howlett's calculation tells us that $H_2(G)$ is an elementary 2--group of rank one, hence $H_2(G)\cong\Z_2$.
	
	Suppose for a contradiction that $\Pp$ has order-2 meridians.
	It follows that $G$ is a free product with amalgamation:
	$$G\cong \pi(\Pp)\ast_{\Z_2}\pi(\Uu').$$
	We can apply the Mayer-Vietoris sequence for group homology~\cite[Corollary~7.7]{Bro_94_Cohomology-of-groups}, which states:
	If $G = G_1\ast_AG_2$, then there is a long exact sequence in group homology:
	$$\cdots \longrightarrow H_2(A) \longrightarrow H_2(G_1)\oplus H_2(G_2) \longrightarrow H_2(G) \longrightarrow H_1(A) \longrightarrow H_1(G_1)\oplus H_1(G_2) \longrightarrow H_1(G) \longrightarrow \cdots$$
	Let $G_1 = \pi(\Pp)$ and $G_2 = \pi(\Uu')$.
	
	First, note that $H_2(G_1)\cong 0$, since $H_2(S^4\setminus\nu(\Pp))\cong 0$ and $H_2(G_1)$ is a quotient of this group.
	Next, note that $H_2(G_2)\cong 0$, since $G_2\cong\Z_2$.
	Finally, note that each of $G$, $G_1$, and $G_2$ abelianize to $\Z_2$.
	It follows that the Mayer-Vietoris sequence collapses to:
	$$\cdots\longrightarrow 0 \longrightarrow H_2(G) \longrightarrow \Z_2 \stackrel{f}{\longrightarrow} \Z_2\oplus\Z_2 \longrightarrow \Z_2 \longrightarrow \cdots$$
	The map $f$ is injective, as it maps the generator of $\Z_2$ to the class of the meridian in each of the $G_i$.
	Therefore, we have $H_2(G)\cong 0$.
	This contradicts the fact $H_2(G)\cong\Z_2$, as computed above.
\end{proof}

\section{Future work}
\label{sec:future}

As mentioned in the introduction, we conjecture the following.

\begin{conjecture}
\label{conj:order-4}
	Each $\Kk(l,m,n)$ has order-4 meridians, is indecomposable, and is stably irreducible. 
\end{conjecture}

Our Klein bottles shared enough properties with the ribbon Klein bottles of Lidman and Piccirillo for us to apply their obstruction to decomposability, but we do not expect that our Klein bottles are ribbon.

\begin{conjecture}
\label{conj:ribbon}
	The $\Kk(l,m,n)$ are not ribbon.
\end{conjecture}

The Klein bottle $\Yy$ with order-4 meridians constructed by Yoshikawa is built by attaching a disoriented tube to the 2--twist spin of the pretzel knot $P(-2,3,3)$~\cite[Section~4]{Yos_98_The-order-of-a-meridian-of-a-knotted}.
Since $\pi(\Yy)$ agrees with $\pi(\Kk(-2,3,3))$, it seems likely that $\Yy$ is isotopic to $\Kk(-2,3,3)$ (up to changing the signs of the parameters; see Question~\ref{ques:signs}), which motivates the following question.

\begin{question}
\label{ques:Yoshi}
	Which of the $\Kk(l,m,n)$ arise as the result of attaching a disoriented tube to a (twist-spun) 2--knot? 
\end{question}

The fact that $\pi(\Kk(l,m,n))$ doesn't depend on the sign of the parameters (see the proof of Proposition~\ref{prop:group_K}) motivates our final question.

\begin{question}
\label{ques:signs}
	Does the isotopy class of $\Kk(l,m,n)$ depend on the signs of the parameters?
\end{question}


\bibliographystyle{amsalpha}
\bibliography{Klein_bottles.bib}

\end{document}